\documentclass[12pt]{article}

\usepackage{verbatim}
\usepackage{amssymb}
\usepackage{amsbsy}
\usepackage{amscd}
\usepackage{amsmath}
\usepackage{amsthm}
\usepackage[mathscr]{eucal}

\numberwithin{defn}{section}

\theoremstyle{plain}
\newtheorem{Thm}{Theorem}[section]

\theoremstyle{definition}

\newtheorem{Rem}[Thm]{Remark}

\newcommand{\bR}{\ensuremath{\mathbb{R}}}

\newcommand{\cC}{\ensuremath{\mathcal{C}}}

\newcommand{\cF}{\ensuremath{\mathcal{F}}}

\newcommand{\vn}{\ensuremath{\mbox{{\boldmath $n$}}}}

\numberwithin{equation}{section}

\makeatletter
\renewcommand\section{\@startsection {section}{1}{\z@}%
                                   {-3.5ex \@plus -1ex \@minus -.2ex}%
                                   {2.3ex \@plus.2ex}%
                                   {\normalfont\large\bf}}
\makeatother

\newcommand{\eps}{\ensuremath{\varepsilon}}
\newcommand{\e}{{\rm e}}
\renewcommand{\d}{{\rm d}}

\newcommand{\law}{\stackrel{{\rm law}}{=}}

\newcommand{\absol}[1]{\left| #1 \right|} 
\newcommand{\rbra}[1]{\left( #1 \right)} 
\newcommand{\cbra}[1]{\left\{ #1 \right\}} 
\newcommand{\sbra}[1]{\left[ #1 \right]} 

\newcommand{\uL}{\underline{L}}
\newcommand{\uR}{\underline{R}}
\newcommand{\uX}{\underline{X}}

\begin{document}

\begin{center}
{\Large \bf    
Two kinds of conditionings for stable L\'evy processes
}
\end{center}
\begin{center}
Kouji \textsc{Yano}
\end{center}

\footnotetext{
Department of Mathematics, Graduate School of Science,
Kobe University, Kobe, Japan.}
\footnotetext{
Research supported by KAKENHI (20740060)}

\bigskip

\begin{abstract}
Two kinds of conditionings 
for one-dimensional stable L\'evy processes are discussed 
via $ h $-transforms of excursion measures: 
One is to stay positive, and the other is to avoid the origin. 
\end{abstract}

\section{Introduction}

It is well-known that 
a one-dimensional Brownian motion conditioned to stay positive 
is a three-dimensional Bessel process. 
As an easy consequence, it follows that 
the former conditioned to avoid the origin 
is a symmetrized one of the latter. 

The aim of the present article 
is to give a brief survey on these two different kinds of conditionings 
for one-dimensional stable L\'evy processes 
via $ h $-transforms of It\^o's excursion measures. 

The organization of this article is as follows. 
In Section \ref{sec: BM}, 
we recall the conditionings for Brownian motions. 
In Section \ref{sec: stay positive}, 
we give a review on the conditioning 
for stable L\'evy processes to stay positive. 
In Section \ref{sec: avoid 0}, 
we present results on the conditioning 
for symmetric stable L\'evy processes to avoid the origin.

\section{Conditionings for Brownian motions}
\label{sec: BM}

We recall the conditionings for Brownian motions. 
For the details, see, e.g., 
\cite[\S III.4.3]{MR1011252} 
and 
\cite[\S VI.3 and Chap.XII]{MR1725357}. 

Let $ (X_t) $ denote the coordinate process 
on the space of c\`adl\`ag functions 
and $ (\cF_t) $ its natural filtration. 
Set $ \cF_{\infty } = \sigma(\cup_{t>0} \cF_t) $. 
For $ t \ge 0 $, we write $ \theta_t $ for the shift operator: 
$ X_s \circ \theta_t = X_{t+s} $. 
For $ 0<t<\infty $, a functional $ Z_t $ is called {\em $ \cF_t $-nice} 
if $ Z_t $ is of the form $ Z_t=f(X_{t_1},\ldots,X_{t_n}) $ 
for some $ 0 < t_1 < \ldots < t_n < t $ 
and some continuous function $ f:\bR^n \to \bR $ 
which vanishes at infinity. 

Let $ W_x $ denote the law of the one-dimensional Brownian motion 
starting from $ x \in \bR $.

\subsection{Brownian motions conditioned to stay positive}

For any fixed $ t>0 $, 
we define a probability law $ W^{\uparrow,(t)} $ on $ \cF_t $ as 
\begin{align}
W^{\uparrow,(t)}(\cdot) 
= \frac{\vn^+(\cdot; \ \zeta>t)}{\vn^+(\zeta>t)} 
\label{}
\end{align}
where $ \vn^+ $ stands for 
the {\em excursion measure of the reflecting Brownian motion} 
(see, e.g., \cite[\S III.4.3]{MR1011252} 
and 
\cite[\S \S XII.4]{MR1725357}) 
and $ \zeta $ for the {\em lifetime}. 
The process $ (X_s:s \le t) $ under $ W^{\uparrow,(t)} $ 
is called the {\em Brownian meander}. 
Durrett--Iglehart--Miller \cite[Thm.2.1]{MR0436353} have proved that 
\begin{align}
W^{\uparrow,(t)}[Z_t] 
= \lim_{\eps \to 0+} W_0 \sbra{ Z_t 
\biggm| \forall u \le t , \ X_u \ge -\eps } 
\label{eq: limit of BMt ge -eps}
\end{align}
for any bounded continuous $ \cF_t $-measurable functional $ Z_t $; 
in particular, for any $ \cF_t $-nice functional $ Z_t $. 
We may represent \eqref{eq: limit of BMt ge -eps} symbolically as 
\begin{align}
W^{\uparrow,(t)}(\cdot) 
= W_0 \rbra{ \cdot \biggm| \forall u \le t , \ X_u \ge 0 } ; 
\label{eq: BM cond to stay pos until t}
\end{align}
that is, the Brownian meander 
is the Brownian motion {\em conditioned to stay positive until time $ t $}. 
As another interpretation of \eqref{eq: BM cond to stay pos until t}, 
we have 
\begin{align}
W^{\uparrow,(t)}[Z_t] 
= \lim_{\eps \to 0+} W_0 \sbra{ Z_t \circ \theta_{\eps} 
\biggm| \forall u \le t , \ X_u \circ \theta_{\eps} \ge 0 } 
\label{eq: BM cond to stay pos until t 2}
\end{align}
for any $ \cF_t $-nice functional $ Z_t $. 
The proof of \eqref{eq: BM cond to stay pos until t 2} 
will be given in Theorem \ref{thm: stable cond to stay pos until t} 
in the settings of stable L\'evy processes. 

We write $ W^{\uparrow}_0 $ for the law 
of a {\em three-dimensional Bessel process}, that is, 
the law of the radius 
$ \sqrt{(B^{(1)}_t)^2 + (B^{(2)}_t)^2 + (B^{(3)}_t)^2} $ 
of a three-dimensional Brownian motion $ (B^{(1)}_t,B^{(2)}_t,B^{(3)}_t) $. 
We remark that 
$ W^{\uparrow}_0 $ is locally equivalent to $ \vn^+ $: 
\begin{align}
\d W^{\uparrow}_0|_{\cF_t} = \frac{1}{C^{\uparrow}} X_t \d \vn^+|_{\cF_t} 
\label{}
\end{align}
where $ C^{\uparrow} = \vn^+[X_t] $ is a constant independent of $ t>0 $. 
We say that $ W^{\uparrow}_0 $ is the {\em $ h $-transform 
of the excursion measure $ \vn^+ $ 
with respect to the function $ h(x)=x $}. 
Then it holds (see, e.g., Theorem \ref{thm: stable cond to stay pos}) that 
\begin{align}
W^{\uparrow}_0[Z] 
= \lim_{t \to \infty } W^{\uparrow,(t)}[Z] 
\label{eq: lim of BM as cond to stay pos}
\end{align}
for any $ \cF_t $-nice functional $ Z $ with $ 0<t<\infty $. 
We may represent \eqref{eq: lim of BM as cond to stay pos} symbolically as 
\begin{align}
W^{\uparrow}_0(\cdot) 
= W_0 \rbra{ \cdot \biggm| \forall u , \ X_u \ge 0 } ; 
\label{eq: BM cond to stay pos symb}
\end{align}
that is, $ W^{\uparrow}_0 $ 
is the Brownian motion {\em conditioned to stay positive during the whole time}.

\subsection{Brownian motions conditioned to avoid the origin}

For any fixed $ t>0 $, 
we define a probability law $ W^{\times,(t)} $ on $ \cF_t $ as 
\begin{align}
W^{\times,(t)} = \frac{W^{\uparrow,(t)} + W^{\downarrow,(t)}}{2} 
\label{}
\end{align}
where $ W^{\downarrow,(t)} $ stands for 
the law of the process $ (-X_s:s \le t) $ under $ W^{\uparrow,(t)} $. 
The law $ W^{\times,(t)} $ may be represented as 
\begin{align}
W^{\times,(t)}(\cdot) 
= \frac{\vn(\cdot; \ \zeta>t)}{\vn(\zeta>t)} 
\label{}
\end{align}
where $ \vn $ stands for the {\em excursion measure of the Brownian motion}; 
in fact, $ \vn = \frac{\vn^+ + \vn^-}{2} $ 
where $ \vn^- $ is the image measure of the process $ (-X_t) $ under $ \vn^+ $. 
Immediately from \eqref{eq: BM cond to stay pos until t 2} 
and continuity of paths, we have 
\begin{align}
W^{\times,(t)}[Z_t] 
= \lim_{\eps \to 0+} W_0 \sbra{ Z_t \circ \theta_{\eps} 
\biggm| \forall u \le t , \ X_u \circ \theta_{\eps} \neq 0 } 
\label{eq: limit of BM neq zero}
\end{align}
for any $ \cF_t $-nice functional $ Z_t $. 
We may represent \eqref{eq: limit of BM neq zero} symbolically as 
\begin{align}
W^{\times,(t)}(\cdot) 
= W_0 \rbra{ \cdot \biggm| \forall u \le t , \ X_u \neq 0 } ; 
\label{eq: BM cond to avoid 0 until t}
\end{align}
that is, 
$ W^{\times,(t)} $ is 
the Brownian motion {\em conditioned to avoid the origin until time $ t $}. 

We define 
\begin{align}
W^{\times}_0 = \frac{W^{\uparrow}_0 + W^{\downarrow}_0}{2} 
\label{eq: symmetrized Bessel}
\end{align}
where $ W^{\downarrow}_0 $ stands for 
the law of the process $ (-X_t) $ under $ W^{\uparrow}_0 $. 
In other words, the law $ W^{\uparrow}_0 $ 
is the symmetrization of the three-dimensional Bessel process. 
We also remark that 
$ W^{\times}_0 $ is locally equivalent to $ \vn $: 
\begin{align}
\d W^{\times}_0|_{\cF_t} = \frac{1}{C^{\times}} |X_t| \d \vn|_{\cF_t} 
\label{}
\end{align}
where $ C^{\times} = \vn[|X_t|] $ is a constant independent of $ t>0 $. 
We say that $ W^{\times}_0 $ is the {\em $ h $-transform 
with respect to the function $ h(x)=|x| $}. 
Then it is immediate from \eqref{eq: lim of BM as cond to stay pos} that 
\begin{align}
W^{\times}_0[Z] 
= \lim_{t \to \infty } W^{\times,(t)}[Z] 
\label{eq: limit B meander t to infty}
\end{align}
for any $ \cF_t $-nice functional $ Z $ with $ 0<t<\infty $. 
We may represent \eqref{eq: limit B meander t to infty} symbolically as 
\begin{align}
W^{\times}_0(\cdot) 
= W_0 \rbra{ \cdot \biggm| \forall u , \ X_u \neq 0 } ; 
\label{eq: BM cond to avoid zero}
\end{align}
that is, $ W^{\times}_0 $ is 
the Brownian motion {\em conditioned to avoid the origin during the whole time}.

\section{Stable L\'evy processes conditioned to stay positive}
\label{sec: stay positive}

Let us review the theory of 
strictly stable L\'evy processes conditioned to stay positive. 
For concise references, 
see, e.g., 
\cite{MR1406564} 
and 
\cite{MR2320889}. 
We refer to these textbooks 
also about the theory of conditioning to stay positive 
for {\em spectrally negative L\'evy processes}, 
where we do not go into the details. 

For a Borel set $ F $, we denote the first hitting time of $ F $ by 
\begin{align}
T_F = \inf \{ t>0 : X_t \in F \} . 
\label{}
\end{align}
Define 
\begin{align}
\uX_t = \inf_{s \le t} X_s 
, \quad 
\uR_t = X_t - \uX_t 
\label{}
\end{align}
and call the process $ (\uR_t) $ the {\em reflected process}. 

Let $ (P_x) $ denote the law of a {\em strictly stable L\'evy process 
of index $ 0 < \alpha \le 2 $}, 
that is, a process with c\`adl\`ag paths and with stationary independent increments 
satisfying the following scaling property: 
\begin{align}
(k^{-\frac{1}{\alpha }}X_{kt}:t \ge 0) \law (X_t:t \ge 0) 
\quad \text{under} \ P_0 
\label{eq: scaling}
\end{align}
for any $ k>0 $. 
Note that the Brownian case corresponds to $ \alpha =2 $. 
From the scaling property \eqref{eq: scaling}, it is immediate that the quantity 
\begin{align}
\rho := P_0(X_t \ge 0) 
\label{}
\end{align}
does not depend on $ t>0 $, which is called the {\em positivity parameter}. 
The possible values of $ \rho $ range over 
$ [0,1] $ if $ 0 < \alpha < 1 $, 
$ (0,1) $ if $ \alpha =1 $, 
and $ [1-\frac{1}{\alpha },\frac{1}{\alpha }] $ if $ 1 < \alpha \le 2 $. 
Let $ (Q_x:x>0) $ denote the law of the process killed at $ T_{(-\infty ,0)} $: 
\begin{align}
Q_x(\Lambda_t ; \ t<\zeta ) = P_x \rbra{ \Lambda_t ; \ t<T_{(-\infty ,0)} } 
, \quad x>0 , \ \Lambda_t \in \cF_t . 
\label{}
\end{align}
Note that the function 
\begin{align}
(x,t) \mapsto Q_x(t<\zeta ) = P_x \rbra{ t<T_{(-\infty ,0)} } 
\label{eq: jt conti Q}
\end{align}
is jointly continuous in $ x>0 $ and $ t>0 $. 

Let us {\bf exclude} the case where $ |X| $ is a subordinator, i.e., 
\begin{align}
0 < \alpha < 1 \quad \text{and} \quad \rho = 0,1 . 
\label{}
\end{align}
Then 
the reflected process $ (\uR_t) $ under $ (P_x) $ is a Feller process 
where the origin is regular for itself, 
and hence 
there exists the continuous local time process $ (\uL_t) $ at level 0 
of the reflected process $ (\uR_t) $ 
such that 
\begin{align}
P_0 \sbra{ \int_0^{\infty } \e^{-qt} \d \uL_t } 
= q^{\rho-1} 
, \quad q>0 . 
\label{}
\end{align}
Let $ \vn^{\uparrow} $ denote the corresponding excursion measure away from 0 
of the reflected process $ (\uR_t) $. 
The Markov property of $ \vn^{\uparrow} $ may be expressed as 
\begin{align}
\vn^{\uparrow} \rbra{ 1_{\Lambda} \circ \theta_t ; \ \Lambda_t , \ t<\zeta } 
= \vn^{\uparrow} \sbra{ Q_{X_t}(\Lambda) ; \ \Lambda_t , \ t < \zeta } 
\label{}
\end{align}
for any $ \Lambda \in \cF_{\infty } $ and $ \Lambda_t \in \cF_t $. 
We introduce the following function: 
\begin{align}
h^{\uparrow}(x) 
= P_0 \sbra{ \int_0^{\infty } 
1_{ \displaystyle \cbra{ \uX_t \ge -x } } \d \uL_t } 
, \quad x \ge 0 . 
\label{}
\end{align}
Since the {\em ladder height process} $ H := \uX \circ \uL^{-1} $ 
is a stable L\'evy process of index $ \alpha (1-\rho) $, we see that 
\begin{align}
h^{\uparrow}(x) 
= P_0 \sbra{ \int_0^{\infty } 
1_{ \displaystyle \cbra{ H_u \ge -x } } \d u } 
= C^{\uparrow}_1 x^{\alpha (1-\rho)} 
\label{}
\end{align}
for some constant $ C^{\uparrow}_1>0 $ independent of $ x \ge 0 $. 
The following theorem is due to 
Silverstein \cite[Thm.2]{MR573292}; 
another proof can be found in Chaumont--Doney \cite[Lem.1]{MR2164035} 
(see also Doney \cite[Lem.10 in \S 8.3]{MR2320889}). 

\begin{Thm}[\cite{MR573292}]
It holds that 
\begin{align}
Q_x[(X_t)^{\alpha (1-\rho)} ; \ t<\zeta] 
=& x^{\alpha (1-\rho)} 
, \quad x>0, \ t>0, 
\\
\vn^{\uparrow}[(X_t)^{\alpha (1-\rho)} ; \ t<\zeta] 
=& C^{\uparrow}_2 
, \quad t>0 
\label{}
\end{align}
for some constant $ C^{\uparrow}_2 $ independent of $ t>0 $. 
\end{Thm}

By virtue of this theorem, we may define the $ h $-transform by 
\begin{align}
\d P^{\uparrow}_x|_{\cF_t} = 
\begin{cases}
\rbra{ \frac{X_t}{x} }^{\alpha (1-\rho)} \d Q_x|_{\cF_t} 
& \text{if} \ x>0 , \\
\frac{1}{C^{\uparrow}_2} (X_t)^{\alpha (1-\rho)} \d \vn^{\uparrow}|_{\cF_t} 
& \text{if} \ x=0 ; 
\end{cases}
\label{eq: def of P uparrow}
\end{align}
indeed, the family $ (P^{\uparrow}_x|_{\cF_t}:t \ge 0) $ 
is proved to be consistent by the Markov property of $ \vn^{\uparrow} $. 

\begin{Thm}[\cite{MR1419491}] \label{thm: Feller uparrow}
The process $ ((X_t),(P^{\uparrow}_x)) $ is a Feller process. 
\end{Thm}

This theorem is due to 
Chaumont \cite[Thm.6]{MR1419491}, 
where he proved weak convergence $ P^{\uparrow}_x \to P^{\uparrow}_0 $ as $ x \to 0+ $ 
in the c\`adl\`ag space equipped with Skorokhod topology. 
Bertoin--Yor \cite[Thm.1]{MR1918243} 
proved the weak convergence for general positive self-similar Markov processes. 
Tanaka \cite[Thm.4]{MR2073341} 
proved the Feller property for quite a general class of L\'evy processes. 

Now let us discuss the conditionings 
\eqref{eq: BM cond to stay pos until t} 
and \eqref{eq: BM cond to stay pos symb} 
for stable L\'evy processes. 
The following theorem is an immediate consequence of Chaumont \cite[Lem.1]{MR1465814} 
and of the continuity of \eqref{eq: jt conti Q}. 

\begin{Thm}[\cite{MR1465814}] \label{thm: phi func uparrow}
Let $ t>0 $ be fixed. 
Then the function 
\begin{align}
[0,\infty ) \ni x \mapsto P^{\uparrow}_x \sbra{ (X_t)^{-\alpha (1-\rho)} } 
\label{}
\end{align}
is continuous and vanishes at infinity. 
\end{Thm}

Define a probability law $ M^{\uparrow,(t)} $ on $ \cF_t $ as 
\begin{align}
M^{\uparrow,(t)}(\Lambda_t) 
= \frac{\vn^{\uparrow}(\Lambda_t ; \ \zeta>t) }{\vn^{\uparrow}(\zeta>t)} 
= \frac{P^{\uparrow}_0[(X_t)^{-\alpha (1-\rho)}; \ \Lambda_t ]}
{P^{\uparrow}_0[(X_t)^{-\alpha (1-\rho)}]} 
\label{eq: def of meander uparrow}
\end{align}
for $ \Lambda_t \in \cF_t $. 
The following theorem, 
which generalizes \eqref{eq: limit of BMt ge -eps}, 
can be found in 
Bertoin \cite[Thm.VIII.18]{MR1406564}. 

\begin{Thm}[\cite{MR1406564}]
For any $ t>0 $, it holds that 
\begin{align}
M^{\uparrow,(t)}[Z_t] = \lim_{\eps \to 0+} P_0 \sbra{ Z_t 
\biggm| \forall u \le t , \ X_u \ge -\eps } 
\label{}
\end{align}
for any $ \cF_t $-nice functional $ Z_t $. 
\end{Thm}

Now we give the following version 
of \eqref{eq: BM cond to stay pos until t 2} 
for stable L\'evy processes. 

\begin{Thm} \label{thm: stable cond to stay pos until t}
For any $ t>0 $, it holds that 
\begin{align}
M^{\uparrow,(t)}[Z_t] 
= \lim_{\eps \to 0+} P_0 \sbra{ Z_t \circ \theta_{\eps} 
\biggm| \forall u \le t , \ X_u \circ \theta_{\eps} \ge 0 } 
\label{eq: stable cond to stay pos until t}
\end{align}
for any $ \cF_t $-nice functional $ Z_t $. 
\end{Thm}

\begin{proof}
By the Markov property, 
the expectation of the right hand side of \eqref{eq: stable cond to stay pos until t} 
is equal to 
\begin{align}
\frac{P_0 \sbra{ P_{X_{\eps}} \sbra{ Z_t; \ \forall u \le t , \ X_u \ge 0 } } }
{P_0 \sbra{ P_{X_{\eps}} \rbra{ \forall u \le t , \ X_u \ge 0 } } } 
= 
\frac{P_0 \sbra{ (X_{\eps})^{-\gamma} P^{\uparrow}_{X_{\eps}} \sbra{ Z_t (X_t)^{-\gamma} } } }
{P_0 \sbra{ (X_{\eps})^{-\gamma} P^{\uparrow}_{X_{\eps}} \sbra{ (X_t)^{-\gamma} } } } , 
\label{}
\end{align}
where we put $ \gamma = \alpha (1-\rho) $. 
By the scaling property, this is equal to 
\begin{align}
\frac{P_0 \sbra{ (X_1)^{-\gamma} P^{\uparrow}_{\eps^{1/\alpha } X_1} \sbra{ Z_t (X_t)^{-\gamma} } } }
{P_0 \sbra{ (X_1)^{-\gamma} P^{\uparrow}_{\eps^{1/\alpha } X_1} \sbra{ (X_t)^{-\gamma} } } } . 
\label{}
\end{align}
By Theorems \ref{thm: Feller uparrow} and \ref{thm: phi func uparrow} 
and by the dominated convergence theorem, we see that 
this quantity converges as $ \eps \to 0+ $ to 
\begin{align}
\frac{P^{\uparrow}_0 \sbra{ Z_t (X_t)^{-\gamma} } }
{P^{\uparrow}_0 \sbra{ (X_t)^{-\gamma} } } , 
\label{}
\end{align}
which coincides with $ M^{\uparrow,(t)}[Z_t] $ 
by the definition \eqref{eq: def of meander uparrow}. 
\end{proof}

The following theorem 
generalizes \eqref{eq: lim of BM as cond to stay pos}. 

\begin{Thm} \label{thm: stable cond to stay pos}
It holds that 
\begin{align}
P^{\uparrow}_0[Z] 
= \lim_{t \to \infty } M^{\uparrow,(t)}[Z] 
\label{}
\end{align}
for any $ \cF_t $-nice functional $ Z $ with $ 0<t<\infty $. 
\end{Thm}

The proof can be done 
in the same way as Theorem \ref{thm: stable cond to stay pos until t}.

\section{Symmetric stable L\'evy processes conditioned to avoid the origin}
\label{sec: avoid 0}

Let us discuss the conditioning 
for symmetric stable L\'evy processes to avoid the origin. 
This has been introduced by Yano--Yano--Yor \cite{YYY} 
in order to extend some of the {\em penalisation problems} for Brownian motions 
by Roynette--Vallois--Yor \cite{MR2261065}, \cite{RY} and Najnudel--Roynette--Yor \cite{NRY}, 
to symmetric stable L\'evy processes. 

We assume that $ \rho = \frac{1}{2} $, i.e., 
the process $ ((X_t),(P_x)) $ is {\em symmetric}, 
and that the index $ \alpha $ satisfies $ 1<\alpha \le 2 $. 
For simplicity, 
we assume that $ P_0[\e^{i \lambda X_1}] = \e^{-|\lambda|^{\alpha }} $. 
Then the origin is regular for itself, 
and 
there exists the continuous resolvent density: 
\begin{align}
u_q(x) 
= \frac{1}{\pi} \int_0^{\infty } \frac{\cos x \lambda}{q+\lambda^{\alpha }} \d \lambda 
, \quad q>0, \ x \in \bR . 
\label{}
\end{align}
Moreover, there exists the local time process at level 0 
of the process $ (X_t) $, 
which we denote by $ (L_t) $, such that 
\begin{align}
P_0 \sbra{ \int_0^{\infty } \e^{-qt} \d L_t } = u_q(0) 
= \frac{q^{\frac{1}{\alpha }-1}}{\alpha \sin \frac{\pi}{\alpha }} 
, \quad q>0 . 
\label{}
\end{align}
The corresponding excursion measure away from 0 
will be denoted by $ \vn^{\times} $. 
We introduce the following function: 
\begin{align}
h^{\times}(x) = \lim_{q \to 0+} \{ u_q(0)-u_q(x) \} 
= \frac{1}{C^{\times}} |x|^{\alpha -1} 
, \quad x \in \bR ; 
\label{}
\end{align}
see, e.g., \cite[Appendix]{YYY2} for the exact value of the constant $ C^{\times} $. 
Note that the function $ h^{\times}(x) $ may also be represented as 
\begin{align}
h^{\times}(x) = P_0[L_{T_{\{ x \}}}] 
, \quad x \in \bR; 
\label{}
\end{align}
see, e.g., \cite[Lem.V.11]{MR1406564}. 
Let $ (P^0_x) $ denote the law of the process $ ((X_t),(P_x)) $ 
killed at $ T_{\{ 0 \}} $: 
\begin{align}
P^0_x(\Lambda_t; \ t<\zeta) = P_x \rbra{ \Lambda_t; \ t<T_{\{ 0 \}} } 
, \quad x \neq 0 , \ \Lambda_t \in \cF_t . 
\label{}
\end{align}
By \cite[Thm.3.5]{YYY}, we see that the function 
\begin{align}
(x,t) \mapsto P^0_x(t<\zeta ) = P_x \rbra{ t<T_{\{ 0 \}} } 
\label{eq: jt conti P0}
\end{align}
is jointly continuous in $ x \in \bR \setminus \{ 0 \} $ and $ t>0 $. 
The following theorem is due to 
Salminen--Yor \cite[eq.(3)]{MR2409011} 
and Yano--Yano--Yor \cite[Thm.4.7]{YYY}. 

\begin{Thm}[\cite{MR2409011}, \cite{YYY}]
It holds that 
\begin{align}
P^0_x[|X_t|^{\alpha -1}] 
=& |x|^{\alpha -1} 
, \quad x \neq 0, \ t>0, 
\\
\vn^{\times}[|X_t|^{\alpha -1} ; \ t<\zeta] 
=& C^{\times} 
, \quad t>0 . 
\label{}
\end{align}
\end{Thm}

By virtue of this theorem, we may define the $ h $-transform by 
\begin{align}
\d P^{\times}_x|_{\cF_t} = 
\begin{cases}
\absol{ \frac{X_t}{x} }^{\alpha -1} \d P^0_x|_{\cF_t} 
& \text{if} \ x \neq 0 , \\
\frac{1}{C^{\times}} |X_t|^{\alpha -1} \d \vn^{\times}|_{\cF_t} 
& \text{if} \ x=0 ; 
\end{cases}
\label{}
\end{align}
indeed, the family $ (P^{\times}_x|_{\cF_t}:t \ge 0) $ 
is proved to be consistent by the Markov property of $ \vn^{\times} $. 
The following theorem is due to \cite[Thm.1.5]{Y}. 

\begin{Thm}[\cite{Y}] \label{thm: Feller times}
Suppose that $ 1<\alpha <2 $. 
Then the process $ ((X_t),(P^{\times}_x)) $ is a Feller process. 
\end{Thm}

\begin{Rem}
Yano \cite[Thm.1.4 and Cor.1.9]{Y} obtained the following long-time behavior of paths: 
If $ 1<\alpha <2 $, then 
\begin{align}
P^{\times}_0 \rbra{ \limsup_{t \to \infty } X_t 
= \limsup_{t \to \infty } (-X_t) 
= \lim_{t \to \infty } |X_t| = \infty } = 1 . 
\label{eq: long time behavior}
\end{align}
\end{Rem}

\begin{Rem}
In the Brownian case ($ \alpha =2 $), 
the process $ ((X_t),(P^{\times}_x)) $ is {\em not} a Feller process. 
Indeed, $ P^{\times}_0 $ is not irreducible (see \eqref{eq: symmetrized Bessel}). 
Contrary to \eqref{eq: long time behavior}, 
the long-time behavior in this case is as follows: 
\begin{align}
P^{\times}_0 \rbra{ \lim_{t \to \infty } X_t = \infty } = 
P^{\times}_0 \rbra{ \lim_{t \to \infty } X_t = - \infty } = 
\frac{1}{2} . 
\label{}
\end{align}
\end{Rem}

Now let us discuss the conditionings 
\eqref{eq: BM cond to avoid 0 until t} 
and \eqref{eq: BM cond to avoid zero} 
for symmetric stable L\'evy processes. 
The following theorem is an immediate consequence of Yano--Yano--Yor \cite[Lem.4.10]{YYY} 
and of the continuity of \eqref{eq: jt conti P0}. 

\begin{Thm}[\cite{YYY}] \label{thm: phi func times}
Let $ t>0 $ be fixed. 
Then the function 
\begin{align}
\bR \ni x \mapsto P^{\times}_x \sbra{ |X_t|^{-(\alpha -1)} } 
\label{}
\end{align}
is continuous and vanishes at infinity. 
\end{Thm}

Define a probability law $ M^{\times,(t)} $ on $ \cF_t $ as 
\begin{align}
M^{\times,(t)}(\Lambda_t) 
= \frac{\vn^{\times}(\Lambda_t ; \ \zeta>t) }{\vn^{\times}(\zeta>t)} 
= \frac{P^{\times}_0[|X_t|^{-(\alpha -1)}; \ \Lambda_t ]}
{P^{\uparrow}_0[|X_t|^{-(\alpha -1)]}} 
\label{eq: def of meander times}
\end{align}
for $ \Lambda_t \in \cF_t $. 
The following theorem generalizes \eqref{eq: limit of BM neq zero}. 

\begin{Thm} \label{thm: stable cond to avoid zero until t}
For any $ t>0 $, it holds that 
\begin{align}
M^{\times,(t)}[Z_t] 
= \lim_{\eps \to 0+} P_0 \sbra{ Z_t \circ \theta_{\eps} 
\biggm| \forall u \le t , \ X_u \circ \theta_{\eps} \neq 0 } 
\label{}
\end{align}
for any $ \cF_t $-nice functional $ Z_t $. 
\end{Thm}

The proof can be done 
in the same way as Theorem \ref{thm: stable cond to stay pos until t} 
by virtue of Theorems \ref{thm: Feller times} and \ref{thm: phi func times}. 
The following theorem generalizes \eqref{eq: limit B meander t to infty}. 

\begin{Thm}[\cite{YYY}] \label{thm: stable cond to avoid zero}
It holds that 
\begin{align}
P^{\times}_0[Z] 
= \lim_{t \to \infty } M^{\times,(t)}[Z] 
\label{}
\end{align}
for any $ \cF_t $-nice functional $ Z $ with $ 0<t<\infty $. 
\end{Thm}

This is a special case of \cite[Thm.4.9]{YYY}.

\textbf{Acknowledgements.}
The author thanks Professor Marc Yor for valuable comments.

\end{document}